\documentclass[a4paper,12pt]{amsart}
\usepackage{fullpage}
\usepackage{amsmath}
\usepackage{amsfonts}
\usepackage{amssymb}
\usepackage{color}
\usepackage[OT4]{fontenc}
\usepackage[latin2]{inputenc}
\usepackage{graphicx}
\usepackage{epstopdf}
\usepackage{caption}
\usepackage{subcaption}
\usepackage{adjustbox}
\usepackage{hyperref}

\DeclareMathOperator\CC{{\it\mathbb C^n}}

\newtheorem{theorem}{Theorem}

\newtheorem{definition}[theorem]{Definition}
\newtheorem{corollary}[theorem]{Corollary}
\newtheorem{proposition}[theorem]{Proposition}

\newtheorem{remark}[theorem]{Remark}
\title{Differential tests for plurisubharmonic functions and Koch curves}

\author{S\l awomir Dinew}
\address{S\l awomir Dinew\\
Department of Mathematics and Computer Science\\
Jagiellonian University, Poland}\email{slawomir.dinew@im.uj.edu.pl}
\author{\.Zywomir  Dinew}
\address{\.Zywomir Dinew\\
Department of Mathematics and Computer Science\\
Jagiellonian University, Poland} \email{zywomir.dinew@im.uj.edu.pl}

\begin{document}
 
\begin{abstract}
	We study minimum sets of singular plurisubharmonic functions and their relation to upper contact sets. In particular we develop an algorithm checking when a naturally parametrized curve is such a minimum set. The case of  Koch curves is studied in detail. We also study the size of the set of upper non-contact points. We show that this set is always of Lebesgue measure zero thus answering an open problem in the viscosity approach to the complex Monge-Amp\`{e}re equation. Finally, we prove that similarly to the case of convex functions, strictly plurisubharmonic lower tests yield existence of upper tests with a control on the opening.

\end{abstract}
\maketitle
\section{Introduction}Minimum sets of (strictly) plurisubharmonic functions appear naturally is several branches of complex analysis, such as   El Mir's theory of extension of closed positive currents or in the $\overline{\partial}$-Neumann theory (see \cite{DD16} and the references therein). We would  like to point out that they appear also in polynomial convexity (see \cite{Sto07}), where they are called totally real sets, and as the common boundary of two disjoint strongly pseudoconvex domains (see \cite{Sa82}). It is thus  interesting to investigate their structure and potential theoretic properties. For $\mathcal C^2$-smooth strictly plurisubharmonic functions this has been done by Harvey and Wells (see \cite{HW73}). Their result states that such minimum sets are contained in $\mathcal C^1$ totally real subvarieties. In \cite{DD16} we investigated these sets under the weaker condition that the possibly singular plurisubharmonic function has strictly positive Monge-Amp\`ere measure. In particular we showed that  in this setting the Hausdorff dimension can be larger than in the case considered by Harvey and Wells. The set itself can also contain nontrivial analytic subvarieties.

In the current note we continue our study initiated in \cite{DD16}. Our aim is twofold: first we  provide more explicit examples of fractal minimum sets and develop an algorithm for producing these. The second goal is to use them in problems related to viscosity theory for the complex Monge-Amp\`ere equation providing thus yet another application of the theory.

In \cite{DD16} we showed an example of a fractal Julia set which is of Hausdorff dimension strictly larger than one but is nevertheless a minimum set of a strictly subharmonic function. This raises a question how generic such examples are.

Even in the planar case, there are no known direct methods how to construct  a strictly subharmonic function that obtains a minimum precisely on a given set, unless the Green function is explicitly known, which is rarely the case. Therefore we propose a rather involved procedure, summarized in Theorem \ref{procedure1}, which can be applied in fairly broad range of  Jordan curves. 

The procedure goes as follows:
\begin{itemize}
	\item Obtain a natural parametrization of the set with strong control of constants. The parametrization being bi-H\"older with the H\"older exponent encoding the Hausdorff dimension;
	\item Such a parametrization forces the set to satisfy the {\it Ahlfors three point condition} with an estimate on the constant (and hence the set is a quasiconformal circle);
	\item  Use this estimate to obtain H\"older continuity up to the boundary of the conformal map from the interior of the quasiconformal circle to the unit disc (e.g. pass from quasiconformality to conformality). The same can be done for the corresponding exteriors;   
	\item Express the Green function by means of the conformal mapping, with control of the decay at the boundary;
	\item Having this, use the criterion from \cite{DD16} to decide whether the set is  a minimum set of a strictly subharmonic function and construct one.
\end{itemize}
Similar ideas were used in \cite{GH99}. The aforementioned criterion is:
\begin{theorem}[\cite{DD16}]\label{1dim}
	Let $K\subset\mathbb C$ be a regular, compact, connected set, with empty interior and not disconnecting the plane, satisfying $LS\left (\alpha\right)$ (\L ojasiewicz-Siciak condition) for $\alpha<2$. Then it is a minimum set of a global  strictly subharmonic function. If the set disconnects the plane, being topologically a circle, then the strictly subharmonic function is not global but can be defined on $\mathbb C\setminus\{a\}$ where $a$ is a point in the bounded component of the complement of $K$.

	If in turn for some point $w\in K$ one has $V_K\left (z\right)=O\left (|z-w|^{\alpha}\right)$ for
	$\alpha>2$, then $K$ cannot be a minimum set of any strictly subharmonic function.
\end{theorem} 

 The main problem when using Theorem \ref{1dim} is how to check the \L ojasiewicz-Siciak condition without knowing the Green function explicitly in advance. This is why  one has to be very careful when performing estimates in our procedure.
 
 In Theorem \ref{1dim} we typically  think of sets $K$ of fractal type or, for example, of Brownian trees or dendrites.
 \begin{figure}[h]
 	\includegraphics[scale=0.5]{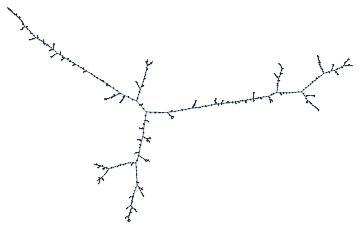}
 	\caption*{a  dendrite}
 \end{figure}

Next we provide such a natural bi-H\"older parametrization of the family of Koch curves (Theorem \ref{kochparametrization}). This improves the corresponding estimates  form \cite{Po93}, using similar ideas. The Koch curves are among the best-known fractal curves, providing (counter)examples to many problems in measure theory and dynamics and therefore our parametrization may find further applications.  

The second main theme in the note are the {\it upper contact sets} that appear in viscosity theory. Recall that these are the points $x_0$ where a fixed plurisubharmonic function $u$ admits a $\mathcal C^2$ differential test from above i.e. a local $\mathcal C^2$ smooth function $q$ such that $q(x)\geq u(x)$ with equality at $x_0$. 

The viscosity method was introduced in \cite{CL83} to deal with Hamilton-Jacobi equations and initial developments focused on first-order equations. Soon it was realized that the methods apply also to second-order nonlinear PDE's, but still only in the real setting. Recently  there has been much interest in applying these techniques in the complex setting with some spectacular results (see \cite{EGZ11}).

The viscosity theory intuitively boils down to exchanging the possibly singular function $u$ by its smooth majorant thus working with genuinely smooth objects. The price to pay is that at points where $u$ does not admit differential tests (the {\it non-contact set}) one has no control. It is thus desirable to prove that the non-contact set is small.

It is easy to see (see \cite{Zer13}) that the contact set is always dense. The measure theoretic properties of the non-contact sets are substantially subtler. In fact, one  of the open problems  posed in \cite{DGZ16} (Question $39$ there), is to characterize the set $E$ on which a plurisubharmonic function allows differential tests from above. It is trivial to see that such a set can be non-(pluri)polar (consider the function $\max\{|z|,1\}$- here the unit circle is the upper non-contact set). The measure theoretic  properties of the upper non-contact sets  are, however, largely unknown. 

We prove  that the upper non-contact set can have big Hausdorff dimension (Theorem \ref{difftest}, the  examples being produced from Koch curves), yet its Lebesgue  measure is always zero (Theorem \ref{trudinger}). It remains an interesting problem to compute the maximal Hausdorff dimension that can occur among these sets.

Finally we investigate when an upper-differential test exist if no smoothness assumptions are posed. Below we recall an important fact in the real theory of the Monge-Amp\`ere operator proven by Caffarelli (\cite{Caf90}, see also \cite{Gut}, Lemma 6.2.1).

\begin{theorem} Let $\Omega$ be a domain in $\mathbb R^n$ and $v$ be a convex function solving the real Monge-Amp\`ere equation
	$$\det(D^2v)=f\geq 0$$
	in $\Omega$ with bounded right hand side $f$.  Suppose also that for some $x_0\in\Omega$ one has $v(x_0)=0$ and there is a constant $\sigma>0$, such that
	$$ u(x)\geq \sigma||x-x_0||^2,$$
	for any $x\in\Omega$.
	Then there  is a constant $\tau$ depending on $\sigma$, $n$ and the uniform bound of $f$, such that
	$$v(x)\leq \tau||x-x_0||^2$$
	for any $x$ sufficiently close to $x_0$.
\end{theorem}

We wish to emphasize that this theorem has important implications in the study of the regularity of solutions to the real Monge-Amp\`ere equation (see \cite{Caf90}, \cite{Gut}).

In the complex setting a lot of the real tools are not available. In particular the Aleksandrov maximum principle and the theory of sections (which are crucial in the proof) are missing. Nevertheless we prove the analogue of this result in the complex setting (Theorem \ref{blocki}).

	 \section{Preliminaries}
 
\subsection{Notions related to holomorphic functions in the plane }
\begin{definition}
	We say that a holomorphic function  $f:\Omega\to \mathbb C$ on a bounded domain $\Omega$ is H\"older continuous up to the boundary with exponent $\alpha$ ($f\in Lip(\alpha)$) if it extends continuously to the boundary of $\Omega$, the extension being also denoted by $f$, and one can find a constant $C$ such that for any $z,w\in \overline{\Omega}$ it holds  $|f(z)-f(w)|<C|z-w|^{\alpha}$. If $\Omega$ is unbounded but $\partial\Omega$ is compact, then the inequality is understood ''for $z$ and $w$ finite'', that is, for any bounded $U\supset\partial \Omega$ one can find a constant $C$ such that $|f(z)-f(w)|<C|z-w|^{\alpha}$, for any $z,w\in U\cap \overline{\Omega}$.
	\end{definition}
	
	It is clear that if $f\in Lip(\alpha)$, then also  $f\in Lip(\alpha')$, for $0<\alpha'<\alpha$.
	\begin{definition}
		Let $\Gamma$ be a  Jordan curve in the complex plane. It is said to satisfy the Ahlfors three point condition with constant $c\geq1$ if  there exist a constant  $\delta>0$, such that for any two points $z_1,\ z_3\in\Gamma$, $|z_1-z_3|\leq\delta$ and any point $z_2$ lying on the the arc joining $z_1$ and $z_3$ having smaller diameter (since one can not claim smaller length due to possible non-rectifiability) there holds
		$$|z_1-z_2|+|z_2-z_3|\leq c|z_1-z_3|.$$
	\end{definition}
		It is known that this property is equivalent to $\Gamma$ being an image of the unit circle by a  quasiconformal mapping of the plane. The constant $c$ depends on $\delta$ and $c=c(\delta)$ is a non decreasing function. Clearly, if $\Gamma$ is a smooth curve then $\lim_{\delta\to0^{+}}c(\delta)=1$. On the other hand, the presence of a corner in $\Gamma$ implies that $c>1$.  As we will see later, also some non rectifiable curves satisfy the Ahlfors three point condition with some finite $c$. The Ahlfors three point condition for Jordan arcs is defined analogously.

\begin{theorem}[Lesley\cite{Le82}]\label{thm1}
 Let $\Gamma$ be a Jordan curve in the complex plane, satisfying the Ahlfors three point condition with constant $c$. 
Then if $f$ and $f^{*}$ are the conformal maps from the unit disc onto the bounded domain bounded by $\Gamma$ (respectively from the complement of the unit disc onto the unbounded domain bounded by $\Gamma$) then the Carath\'eodory extensions of $f, f^{*}, f^{-1}$ and $(f^{*})^{-1}$ are H\"older continuous up to the corresponding boundaries with the following H\"older exponents:
\begin{equation*}
 f,\ f^*\in Lip\left(\frac{2\arcsin^2(c^{-1})}{\pi^2-\pi \arcsin(c^{-1})}\right);
\end{equation*}

\begin{equation*}
 f^{-1},  (f^*)^{-1}\in Lip\left(\frac{\pi}{2\pi-2\arcsin(c^{-1})}\right).
\end{equation*}

\end{theorem}

Actually, Lesley proved Theorem \ref{thm1} only for $f$ and $f^{-1}$ but the same argument applies to $f^{*}$ and $(f^{*})^{-1}$ (compare \cite{Le83}).
Note, that by an earlier paper by Lesley (\cite{Le79}, see also \cite{Be89} for a slightly more general result), under the above conditions on $\Gamma$, if $f\in Lip(\alpha)$ then $(f^*)^{-1}\in Lip(\frac{1}{2-\alpha})$,
and if $f^{*}\in Lip(\beta)$ then $(f)^{-1}\in Lip(\frac{1}{2-\beta})$.

The following immediate corollary will be useful in our considerations:
\begin{corollary}
 If $c\to 1^{+}$ for a family of curves, then both H\"older exponents converge to $1$. In particular for $c$ sufficiently close to $1$ all four mappings are
 in $Lip(1-\varepsilon)$ for any preassigned $\varepsilon>0$.
\end{corollary}
\begin{remark}
 Analogous results with conditions on cross-ratios of quadruples of points was proved by N\"akki and Palka \cite{NP80}. We shall use Lesley's
 theorem since the Ahlfors three point condition is easier to check.
\end{remark}

\begin{definition}
	Let $f:U\rightarrow\Omega$ be a homeomorphism between domains in the complex plane. Now $f$ is said to be $K$-quasiconformal for some $K\geq 1$ if for any $z\in U$
	$$\limsup_{r\rightarrow 0^{+}}\frac{max_{|h|=r}|f\left (z+h\right)-f\left (z\right)|}{min_{|h|=r}|f\left (z+h\right)-f\left (z\right)|}\leq K.$$
\end{definition}
\begin{definition} A quasicircle or quasiconformal circle is the image of the unit circle under a $K$-quasiconformal mapping, for some $K\geq 1$.
	\end{definition}
\subsection{Koch curves}Following \cite{Po93}, we define:
\begin{definition}For any $\theta\in(0,\frac{\pi}{4}]$, let $\Delta_{1}^{0}$ be the isosceles triangle with base angle $\theta$ and with   vertices $0,1$ and $\frac{1}{2}+i\frac{1}{2}\tan{\theta}$.
	The {\it Koch curve of angle} $\theta$ is $$\Gamma_{\theta}:=\bigcap_{n=1}^{\infty}\bigcup_{k=1}^{2^{n}}\Delta_{k}^{n},$$ where $\Delta_{k}^{n}$ are defined inductively as one of the two triangles obtained from the predecessor $\Delta^{n-1}_{\lceil\frac{k}{2}\rceil}$ which are similar to  it and the longest side $c_{k}^{n}$ of $\Delta_{k}^{n}$ is one of the lateral sides of  $\Delta^{n-1}_{\lceil\frac{k}{2}\rceil}$.
\end{definition}
\begin{figure}[h]
	\centering
	\begin{subfigure}[b]{0.5\textwidth}
		\includegraphics[width=\textwidth]{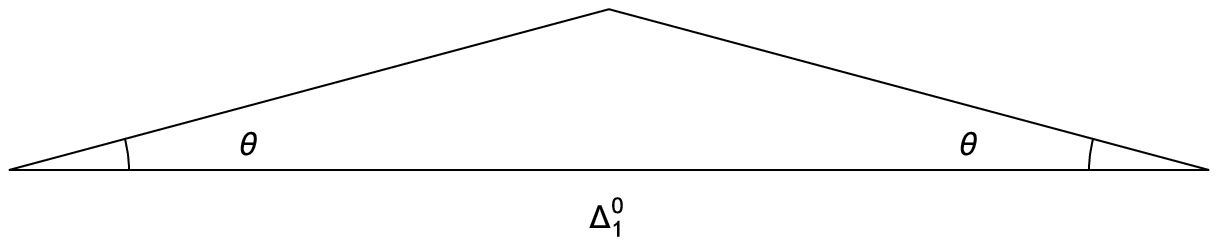}
		\end{subfigure}
	~ 
	\begin{subfigure}[b]{0.5\textwidth}
		\includegraphics[width=\textwidth]{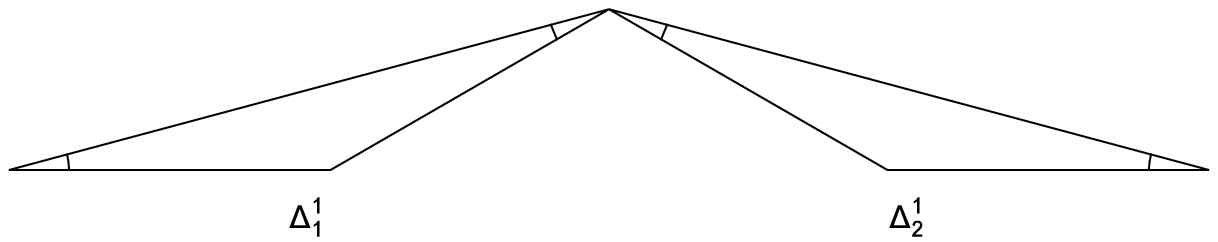}
		\end{subfigure}
	
\end{figure}
\begin{figure}[h]
	\centering
	\begin{subfigure}[b]{0.5\textwidth}
		\includegraphics[width=\textwidth]{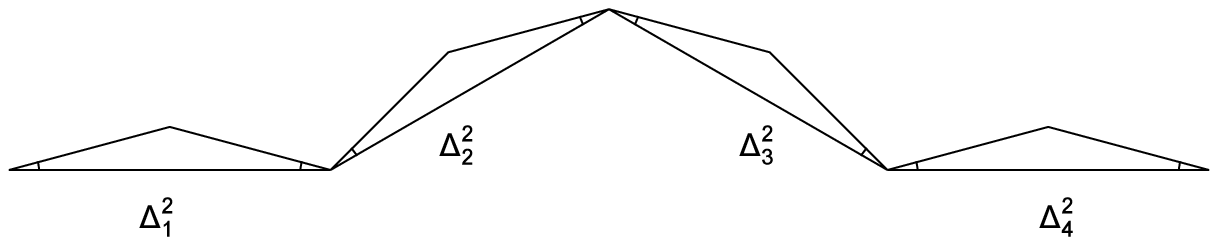}
	\end{subfigure}
	~ 
	\begin{subfigure}[b]{0.5\textwidth}
		\includegraphics[width=\textwidth]{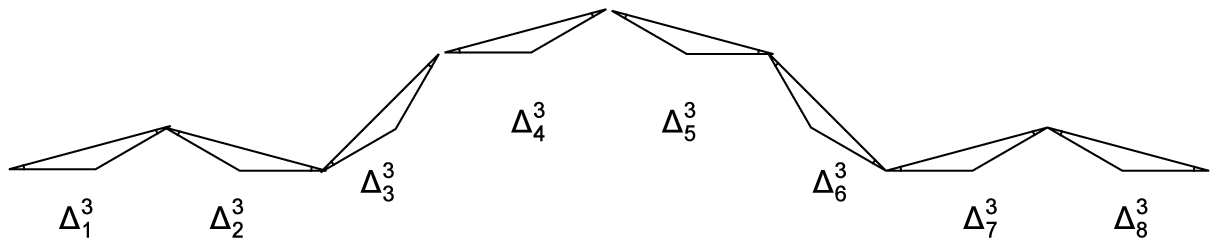}
	\end{subfigure}
	
\end{figure}
\begin{figure}[h]
	\centering
	\begin{subfigure}[b]{1.5\textwidth}
		\adjustbox{trim={.2\width} {.4\height} {0.2\width} {.4\height},clip}%
		{\includegraphics[width=\textwidth]{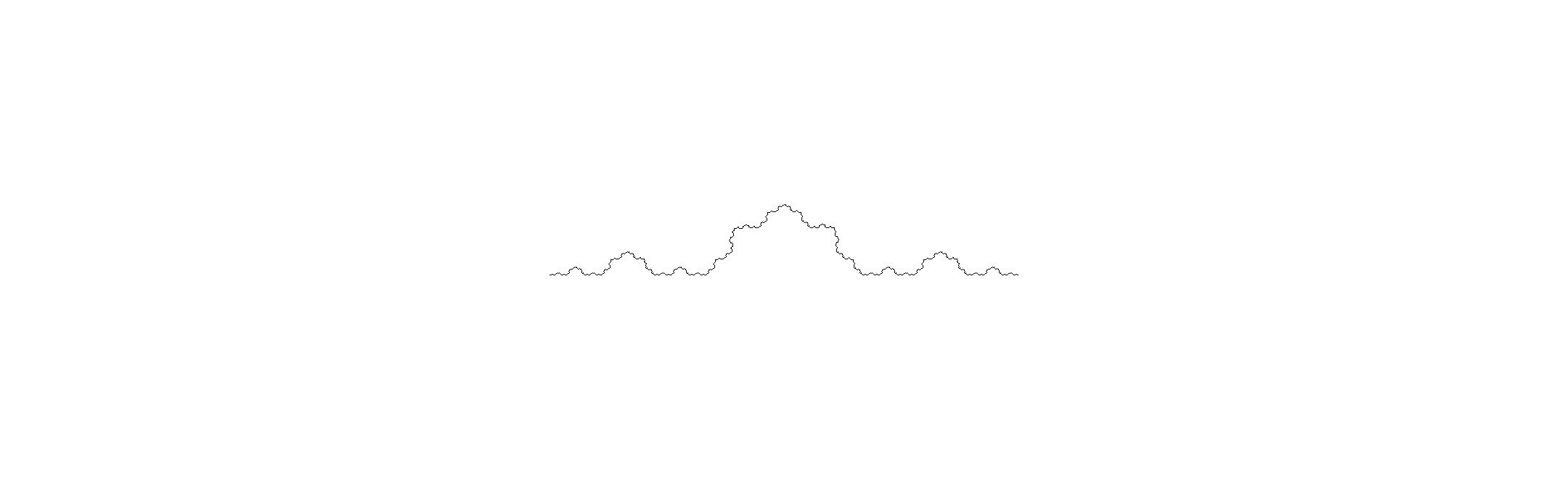}}
		 
	\end{subfigure}
	~ 
	\caption*{Sets approximating $\Gamma_{\theta}$}
\end{figure}

 Although we take the intersection of two dimensional objects, it can be proved that $\Gamma_{\theta}$ is a curve that is identical with the one obtained by the classical construction (usually the Koch curve is given as the limit of an approximating sequence of curves). Moreover, by construction, $\Gamma_{\theta}$ has no self-intersections (see also \cite{CP11}) and hence it is a homeomorphic image of the closed interval $[0,1]$.

An old result of Ponomarev (\cite{Po93}, see also \cite{Po07} for a small correction) states that the Koch curves $\Gamma_{\theta}$ admit a bi-H\"older parametrization 
$$A(\theta)|t_{1}-t_{2}|^{\gamma}\leq |\varphi(t_{1})-\varphi(t_{2})|\leq B(\theta)|t_{1}-t_{2}|^{\gamma},$$
where $\gamma=\gamma(\theta)=\frac{\log(2\cos\theta)}{\log2}$ and

$$A(\theta)=\begin{cases} \frac{1}{4\cos^2(\theta)}\ \ {\rm if}\ 0<\theta<\frac{\pi}8\\
\frac{\sin(3\theta)}{8\cos^3(\theta)}\ \ {\rm if}\ \frac{\pi}{8}\leq\theta<\frac{\pi}{4}
\end{cases},$$
while $B(\theta)=4$.

While sufficient for many practical purposes (see \cite{Po93}) these bounds for the H\"older constants are not asymptotically optimal as $\theta\to 0^+$ and as a result they cannot be used to get realistic \L ojasiewicz-Siciak exponents for the associated Koch curves.

It is well known (see \cite{Po93}) that the Hausdorff dimension of $\Gamma_{\theta}$ is $\frac{\log 2}{\log (2\cos\theta)}=\frac{1}{\gamma}$.

By substituting the edges of a regular $n$-gon,  centered at $0$ and with side lengths equal to $1$, with appropriately rotated and translated $\Gamma_{\theta}$'s   (so that the endpoints coincide with the vertices of the edge) we end up with a Jordan curve which we denote by   $\Pi_{\theta}$. If $n$ is big enough, the neighboring copies of $\Gamma_{\theta}$ will meet up ''at an angle'' of angular measure close enough to $\pi$, so that the  Ahlfors three point constant of $\Pi_{\theta}$ will be the same as for $\Gamma_{\theta}$. 
 Also the Hausdorff dimension of $\Pi_{\theta}$ is the same as of $\Gamma_{\theta}$.
 \subsection{Notions from potential theory}
  
 Denote by $V_{K}^{*}$ the Siciak-Zahariuta extremal function  (or Green function) with pole at infinity associated to the compact set $K\subset\mathbb C$  (or more generally $K\subset\CC$).
 
  \begin{remark}\label{conformalgreen}If $K$ is the closure of a Jordan domain in $\mathbb C$, then it is known that   $V_K^*=\log |(f^*)^{-1}|$, where $f^{*}$ is the conformal mapping from the complement  unit disc to the complement of the Jordan domain, fixing the point at infinity.
  	\end{remark}
 \begin{definition}
 	A regular (in the sense of pluripotential theory) compact set $K$ is said to satisfy the { \L ojasiewicz-Siciak condition} of order $\alpha$  ($K\in LS\left (\alpha\right)$) if 
 	$V_K=V_K^{*}$
 	satisfies the inequality
 	$$V_K\left (z\right)\geq Cdist\left (z,K\right)^{\alpha},\ {\rm if}\ dist\left (z,K\right)\leq 1$$
 	for some positive constant $C$ independent of the point $z$. The distance is with respect to the usual Euclidean metric.
 \end{definition} 
 
 \begin{remark}
 	\label{pierzchala} If the set $K\subset \mathbb C$ has simple geometry - it is connected, does not disconnect the plane and it is locally a smooth curve except for a finite set of points at which branching may occur, and if the finite number of curves at the branching points meet at angles $\theta_{1},\cdots,\theta_{k}$, then $K\in L S(\alpha)$ with $\alpha=\max\{1, (\min\{\frac{\theta_{1}}{\pi},\cdots,\frac{\theta_{k}}{\pi}\})^{-1}\}$ (see \cite{Pi14}).
 \end{remark}
 \begin{corollary} To use Theorem \ref{1dim}, all the angles at which curves meet have to be obtuse. In particular there should be no point at which more than three curves meet.\end{corollary}
\subsection{Viscosity notions}

\begin{definition} A function $q$ defined on some neighborhood $V$ of a point $z$ is called a differential test from above at $z$ for the upper-semicontinuous function $\varphi$ defined on a domain $\Omega\subset \CC$,  also containing $z$, if  it is $\mathcal C^{2}$ smooth on $V$, and $$\varphi(z)-q(z)=\sup_{w\in V\cap\Omega}(\varphi(w)-q(w)).$$
\end{definition}
Note that if $q\in\mathcal C^{2}(V)$, $q\geq \varphi$ on $V\cap\Omega$ and $\{w\in\Omega\cap V| q(w)=\varphi(w)\}\ni z$ then $q$ is a  differential test from above for  $\varphi$ at $z$.

\begin{definition}
	An upper-semicontinuous function $\varphi$ on a domain $\Omega$ is said to allow a differential test from above at $z\in\Omega$ if there exists $V\ni z$ such that the set of differential tests from above for $\varphi$ at $z$ is non-empty.  
\end{definition}

Clearly $\varphi$ allows a differential test from above at any point at which it is twice differentiable. 
 
One should consult the notes \cite{Zer13} for a good introduction to the viscosity theory from the complex analysis perspective.

\section{Results}

\subsection{Jordan curves in $\mathbb C$}
The following easy proposition yields the right constant in Lesley's theorem in the case of a bi-H\"older parametrization:
\begin{proposition}\label{przejscie}
	Suppose that a curve $\Gamma$ is parametrized by a parametrization $\varphi$ satisfying
	$$A|t-s|^{\gamma}\leq|\varphi(t)-\varphi(s)|\leq B|t-s|^{\gamma},$$
	for some $\gamma\leq 1$.
	Then $\Gamma$ satisfies the Ahlfors condition with the constant $c=\frac{B2^{1-\gamma}}{A}$.
\end{proposition}
\begin{proof}
	Fix any three points $z_1=\varphi(t_1),\ z_2=\varphi(t_2),\ z_3=\varphi(t_3)$ on $\Gamma$ with $z_2$ in between $z_1$ and $z_3$ (i.e.
	$t_1<t_2<t_3$).
	Then 
	\begin{align*}
	&|z_1-z_2|+|z_2-z_3|=|\varphi(t_1)-\varphi(t_2)|+|\varphi(t_2)-\varphi(t_3)|\\
	&\leq B(|t_1-t_2|^{\gamma}+|t_2-t_3|^{\gamma})\leq 2B\left(\frac{|t_1-t_2|+|t_2-t_3|}{2}\right)^{\gamma}=2^{1-\gamma}B|t_1-t_3|^{\gamma}\\
	&\leq \frac{2^{1-\gamma}B}{A}|\varphi(t_1)-\varphi(t_3)|=\frac{2^{1-\gamma}B}{A}|z_1-z_3|.
	\end{align*}
	
\end{proof}
 \begin{theorem}\label{procedure1} Any Jordan curve that can be parametrized by a  parametrization $\varphi$ such that
 		$$A|t-s|^{\gamma}\leq|\varphi(t)-\varphi(s)|\leq B|t-s|^{\gamma},$$ for some $\gamma\leq 1$  is a minimum set of a strictly subharmonic function defined on  $\mathbb C\setminus\{a\}$ for some $a$ in the bounded component of the complement of the curve, provided that:
 		
 		\begin{equation}
 		\label{procedure}\frac
 		 {2^{1-\gamma}B}{A}<\frac{1}{\sin\left(\frac{1}{8} \left(\sqrt{17}-1\right) \pi\right)}\approx 1.06237.\end{equation}
 	\end{theorem}
 	\begin{proof} Theorem \ref{1dim} ensures that it is enough to prove that if $K$ is the closure of the bounded component of the complement of the curve, then ${K}\in LS(\alpha)$ for some $\alpha<2$. This will be enough to obtain that the curve is one-sided minimum, in the unbounded component. A conformal transformation sending $a$ to the point at infinity shows that the same  will hold inside the bounded component. One can use a quasiconformal reflection with respect to the curve as well, as in \cite{DD16}.
 		
 		Now $K\in LS(\alpha)$,  by Remark \ref{conformalgreen}, is the same as $\log |(f^{*})^{-1}(z)|\geq C|dist(z,K)|^{\alpha} $. By proposition $14$ in \cite{DD16}, this reduces to $f^{*}\in Lip(\frac{1}{\alpha})$ up to the boundary. By Lesley's theorem  we have that $\alpha=\frac {\pi^2-\pi \arcsin(c^{-1})}{2\arcsin^2(c^{-1})}$. The function $f(c)= \frac {\pi^2-\pi \arcsin(c^{-1})}{2\arcsin^2(c^{-1})}$ is increasing for $c\geq 1$, and the solution to the equation $f(c)=2$ is $c=\frac{1}{\sin\left(\frac{1}{8} \left(\sqrt{17}-1\right) \pi\right)}$. Finally Proposition \ref{przejscie} gives that $c=\frac{2^{1-\gamma}B}{A}$. Hence we get $\frac{\pi^2-\pi \arcsin(\frac{A}{2^{1-\gamma}B} )}{2\arcsin^2(\frac{A}{2^{1-\gamma}B})}<2$ if $A,B$ and $\gamma$ satisfy (\ref{procedure}).
 		\end{proof}
 		\begin{remark}
 			There are many instances where a bi-H\"older parametrization is known to exist. This is the case when $K$ is the invariant set of an iterated function system of contracting similarities (see \cite{IW15}) or the boundary of some self-affine tiles (see \cite{AL11}). In particular, besides  Koch curves which are explicitly dealt with below, many other known fractals such as the boundary of the Rauzy fractal, the Gosper island, various dragon curves, enjoy this property. On the other hand every quasiconformal curve is a bi-Lipshitz image (thus not destroying the bi-H\"olderness of the parametrization) of a member of a family of curves which are essentially slight generalizations of the  Koch curve (see \cite{Ro01}). Of course a necessary condition is that the Hausdorff dimension is constant in any piece of the curve.
 		\end{remark}
 		
Below we prove that the parametrization studied by Ponomarev admits slightly better constants $A(\theta)$ and $B(\theta)$. It turns out that these are good enough for our purposes as long as $\theta$ is sufficiently small. Our proof is inspired by an argument form \cite{BJPP97}:

\begin{theorem}\label{kochparametrization}
	The natural parametrization $\varphi$ of the curve $\Gamma_{\theta}$ satisfies the two-sided  condition:
	\begin{equation}\label{1}
	A(\theta)|t_{1}-t_{2}|^{\gamma}\leq |\varphi(t_{1})-\varphi(t_{2})|\leq B(\theta)|t_{1}-t_{2}|^{\gamma},
	\end{equation}
	where $\gamma=\gamma(\theta)=\frac{\log (2\cos\theta)}{\log2}$ and \begin{equation}A(\theta)=\begin{cases} \max\{\cos2\theta\cos\theta-\frac{8\cos^{2}\theta\sin\theta}{2\cos\theta-1},\frac{1}{4\cos^2(\theta)}\}\ \ {\rm if}\ 0<\theta<\frac{\pi}8\\ \frac{\sin(3\theta)}{8\cos^3(\theta)}\ \ {\rm if}\ \frac{\pi}{8}\leq\theta<\frac{\pi}{4}\end{cases}\end{equation}
	\begin{equation}B(\theta)=\frac{1}{\cos\theta}+\frac{8\cos^{2}\theta\sin\theta}{2\cos\theta-1}.\end{equation}
	In particular $\lim_{\theta\to0^{+}}A(\theta)=\lim_{\theta\to0^{+}}B(\theta)=\lim_{\theta\to0^{+}}\gamma(\theta)=1$.
\end{theorem}
\begin{remark}
	  For small $\theta$, the value of $A(\theta)$ is $\cos2\theta\cos\theta-\frac{8\cos^{2}\theta\sin\theta}{2\cos\theta-1}$ which is a decreasing function of $\theta$. For $\theta >\theta_{0}\approx0.091$ the dominating term is $\frac{1}{4\cos^2(\theta)}$. Actually $\cos2\theta\cos\theta-\frac{8\cos^{2}\theta\sin\theta}{2\cos\theta-1}<0$, for $\theta>\theta_{1}\approx0.121$.
	\end{remark}
\begin{proof}
 
	The values for $A(\theta)$, except $\cos2\theta\cos\theta-\frac{8\cos^{2}\theta\sin\theta}{2\cos\theta-1}$ are by Ponomarev, \cite{Po07}.
	
	Let $\Gamma_{\theta}^{n}:=\bigcup_{k=1}^{2^{n}}c_{k}^{n}$ be the piecewise linear curve approximating $\Gamma_{\theta}$ as $n\to\infty$. Observe that the natural parametrization $\varphi^{n}$ (the piecewise linear map from $[0,1]$ to $\Gamma_{\theta}^{n}$ sending $\frac{k}{2^{n}}, k=0,\cdots,2^{n}$ to the consecutive nodes of $\Gamma_{\theta}^{n}$) can also be described in the following manner:
	
	Fix $x\in(0,1)$ and let $\varphi^{1}(x)$ be the second intersection point of the line through $x$, perpendicular to $c^{0}_{1}$, with $\Delta^{0}_{1}$. If $\varphi^{1}(x)$ is on $c_{1}^{1}$, draw a line perpendicular to $c_{1}^{1}$ through $\varphi^{1}(x)$ and let $\varphi^{2}(x)$ be its second intersection point with $\Delta_{1}^{1}$; if $\varphi^{1}$ is on $c^{1}_{2}$ proceed accordingly. If $\varphi^{1}(x)$ is the common vertex of $\Delta_{1}^{1}$ and $\Delta_{2}^{1}$ then leave $\varphi^{2}(x)= \varphi^{1}(x)$. Finally define $\varphi^{j}(x)$ from $\varphi^{j-1}(x)$ inductively.

			\begin{figure}[h]
				\centering
				\begin{subfigure}[b]{0.5\textwidth}
					\includegraphics[width=\textwidth]{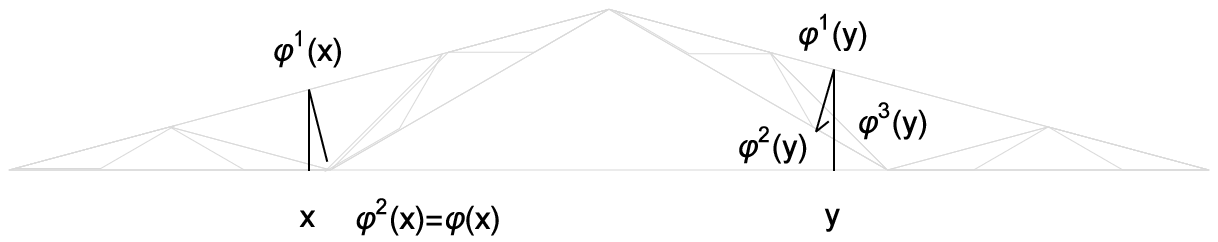}
				\end{subfigure}
			\end{figure}
	 For $\varphi(x)$ we take $\lim_{n\to\infty}\varphi^{n}(x)$ (the limit exists for any $x\in[0,1]$, since $\varphi^{n}(x)$ is a Cauchy sequence for any fixed $x$, as it will become clear from the argument below).
	 
	 Fix now $x,y\in[0,1]$.  Just as in \cite{Po93}, fix $m$ ($m$ may be zero) such that $\varphi^{j}(x)$ and $\varphi^{j}(y)$ belong to the same segment $c_{s_{j}}^{j}, j=0,\cdots,m$ but $\varphi^{m+1}(x)$ and $\varphi^{m+1}(y)$ are already on different segments $c_{s_{m+1}}^{m+1}$ and $c_{s_{m+1}+1}^{m+1}$ respectively. Let also $n\geq m+1$ be the least number such that $\varphi^{j}(x)$ and $\varphi^{j}(y), j=m+1,\cdots,n$ are on neighboring segments $c_{s_{j}}^{j}$ and $c_{s_{j}+1}^{j}$, while  $\varphi^{n+1}(x)$ and $\varphi^{n+1}(y)$ belong to segments $c_{k}^{n+1}$ and $c_{l}^{n+1}$ that do not share a common endpoint. We note trivially that $2\leq |k-l|\leq 3$ in this case and 
	\begin{equation}\label{eq0}\frac{1}{2^{n+1}}\leq |x-y|\leq \frac{1}{2^{n-1}}.\end{equation}
	Note that for any $k$
	\begin{equation}\label{eq00}|\varphi^{k+1}(t)-\varphi^{k}(t)|\leq length(c_{s}^{k+1})\sin\theta=\lambda^{k}\sin\theta,\end{equation}
   with $\lambda=\frac{1}{2\cos\theta}$. Thus	
   
   \begin{equation}\label{eq1}
   |\varphi(x)-\varphi^{n}(x)|=\left|\sum_{j=n}^{\infty}\varphi^{j+1}(x)-\varphi^{j}(x)\right|\leq  \sum_{j=n}^{\infty}\lambda^{j}\sin\theta=\lambda^{n}\frac{\sin\theta}{1-\lambda}.
   \end{equation}
   Observe also that  for $j=1,\cdots,m$, 
   
   $$|\varphi^{j}(x)-\varphi^{j}(y)|=\frac{|\varphi^{j-1}(x)-\varphi^{j-1}(y)|}{\cos\theta}=\cdots=\frac{|x-y|}{\cos^{j}\theta}.$$
   For $j=m+1,\cdots,n$ let $z$ be the common endpoint of the neighboring segments $c_{s_{j}}^{j}$ and $c_{s_{j}+1}^{j}$ containing $\varphi^{j}(x)$ and $\varphi^{j}(y)$. Then
   \begin{equation}\label{eq2}
   |\varphi^{j}(x)-\varphi^{j}(y)|\leq |\varphi^{j}(x)-z|+|\varphi^{j}(y)-z|\leq \frac{|\varphi^{j-1}(x)-z|}{\cos\theta}+\frac{|\varphi^{j-1}(y)-z|}{\cos\theta}=\cdots 
   \end{equation}
   $$= \frac{|\varphi^{m}(x)-\varphi^{m}(\varphi^{-1}(z))|}{\cos^{j-m}\theta}+\frac{|\varphi^{m}(y)-\varphi^{m}(\varphi^{-1}(z))|}{\cos^{j-m}\theta}=\frac{|\varphi^{m}(x)-\varphi^{m}(y)|}{\cos^{j-m}\theta}=\frac{|x-y|}{\cos^{j}\theta}.$$
   
   Note that $\varphi^{m}(x),\varphi^{m}(\varphi^{-1}(z))$ and $\varphi^{m}(y)$ belong to the same segment $c_{s_{m}}^{m}$ !

  Coupling (\ref{eq1}) and (\ref{eq2}) we obtain
  
  \begin{equation}
  \label{eq3}
  |\varphi(x)-\varphi(y)|\leq |\varphi^{n}(x)-\varphi^{n}(y)|+2\lambda^{n}\frac{\sin\theta}{1-\lambda}\leq \frac{|x-y|}{\cos^{n}\theta}+2\lambda^{n}\frac{\sin\theta}{1-\lambda}.
  \end{equation}

Note that by (\ref{eq0})

\begin{equation}
\label{eqq4}
-\frac{\log |x-y|}{\log 2}-1\leq n \leq -\frac{\log |x-y|}{\log 2}+1.
\end{equation}

We have $\frac{1}{\cos\theta}>1$ and $\lambda<1$ hence $$\frac{1}{\cos^{n}\theta}\leq \frac{1}{(\cos\theta)^{-\frac{\log |x-y|}{\log 2}+1}}=\frac{|x-y|^{\frac{\log(\cos\theta)}{\log 2}}}{\cos\theta}$$ and 
$$\lambda ^{n}\leq \frac{1}{(2\cos\theta)^{-\frac{\log |x-y|}{\log 2}-1}}=2\cos\theta|x-y|^{\frac{\log(2\cos\theta)}{\log 2}}. $$
Hence we obtain the bound from above  in (\ref{eq3})

$$\frac{|x-y|}{\cos^{n}\theta}+2\lambda^{n}\frac{\sin\theta}{1-\lambda}\leq\frac{1}{\cos\theta}\left|x-y\right|^{\frac{\log (2\cos\theta)}{\log 2}}+\frac{4\cos\theta\sin\theta}{1-\frac{1}{2\cos\theta}}\left|x-y\right|^{\frac{\log (2\cos\theta)}{\log 2}}=B(\theta)|x-y|^{\gamma}$$
and thus the right hand inequality of (\ref{1}) is established.

To get the bound from below consider the angle between the segments 
$[\varphi^{j}(x),z]$ and $[z,\varphi^{j}(y)]$, defined as above, for $j=m+1,\cdots,n$. Two cases may occur:

a) either the angle is $\pi-2\theta$

b) or the angle is $\pi -4\theta$

In the first case,
$$|\varphi^{j}(x)-\varphi^{j}(y)|^2=|\varphi^{j}(x)-z|^2+|\varphi^{j}(y)-z|^2+2 |\varphi^{j}(x)-z||\varphi^{j}(y)-z|\cos2\theta$$  
$$=\left(|\varphi^{j}(x)-z|+|\varphi^{j}(y)-z|\right)^2-2(1-\cos2\theta)|\varphi^{j}(x)-z||\varphi^{j}(y)-z|$$$$\geq\left(|\varphi^{j}(x)-z|+|\varphi^{j}(y)-z|\right)^2-4\sin^{2}\theta\frac{\left(|\varphi^{j}(x)-z|+|\varphi^{j}(y)-z|\right)^2}{4}$$
$$= \left(|\varphi^{j}(x)-z|+|\varphi^{j}(y)-z|\right)^2\cos^{2}\theta=\frac{|x-y|^2\cos^{2}\theta}{\cos^{2j}\theta},$$

so 
\begin{equation}
\label{eq4}
|\varphi^{j}(x)-\varphi^{j}(y)|\geq \frac{|x-y|\cos\theta}{\cos^{j}\theta}.
\end{equation}

In the second case analogous reasoning leads to

$$|\varphi^{j}(x)-\varphi^{j}(y)|\geq \frac{|x-y|\cos2\theta}{\cos^{j}\theta}.$$

Thus in both cases we obtain 

$$|\varphi^{j}(x)-\varphi^{j}(y)|\geq \frac{|x-y|\cos2\theta}{\cos^{j}\theta}.$$

Now by using the triangle inequality and (\ref{eq00}) 

$$|\varphi(x)-\varphi(y)|\geq|\varphi^{n}(x)-\varphi^{n}(y)|-\sum_{j=n}^{\infty}|\varphi^{j+1}(x)-\varphi^{j}(x)|-\sum_{j=n}^{\infty}|\varphi^{j+1}(y)-\varphi^{j}(y)|$$
$$\geq\frac{\cos2\theta|x-y|}{\cos^{n}\theta}-\sum_{j=n}^{\infty}2\lambda^{j}\sin{\theta}=\frac{\cos2\theta|x-y|}{\cos^{n}\theta}-\lambda^{n}\frac{2\sin\theta}{1-\lambda}.$$
Again by (\ref{eqq4}) the latter is bounded below by 

$$\cos2\theta\cos\theta|x-y|^{\frac{\log (2\cos\theta)}{\log 2}}-2\cos\theta|x-y|^{\frac{\log (2\cos\theta)}{\log 2}}\frac{2\sin\theta}{1-\frac{1}{2\cos\theta}}$$
$$=|x-y|^{\frac{\log (2\cos\theta)}{\log 2}}\left(\cos2\theta\cos\theta-\frac{8\cos^{2}\theta\sin\theta}{2\cos\theta-1}\right)=A(\theta)|x-y|^{\gamma}.$$

This establishes the left hand side of (\ref{1}).
	\end{proof}

\begin{corollary}\label{theta}
	The Jordan curve formed of Koch curves, $\Pi_{\theta}$, is a minimum set of a strictly subharmonic function for any $\theta\in (0,\tilde{\theta})$ with $\tilde{\theta}\approx 0.00378$. 
	

\end{corollary} 
\begin{proof} By Theorem \ref{procedure1}, we need to show that $$\frac{2^{1-\gamma}B}{A} =\frac{2^{1-\frac{\log (2\cos\theta)}{\log2}}\left(\frac{1}{\cos\theta}+\frac{8\cos^{2}\theta\sin\theta}{2\cos\theta-1}\right)}{\cos2\theta\cos\theta-\frac{8\cos^{2}\theta\sin\theta}{2\cos\theta-1}}<1.06237.$$
	The expression above, treated as a function of $\theta$ is increasing. Numerically, equality holds for $\theta=0.00378$.
\end{proof}
\begin{proposition}
	If $\Gamma$ is a quasicircle, then it can be the minimum set of a strictly subharmonic function only if its Hausdorff dimension is less than $\frac{10}{9}$. In particular $\Pi_{\theta}$ can not be a minimum set of a strictly subharmonic function for any $\theta>\arccos(2^{-\frac{1}{10}})\approx 0.368044$.
\end{proposition}	
	\begin{proof}
	If $\Gamma$ is a quasicircle, then it is well known that $f^*$ and $f$ can be extended to a $K$-quasiconformal mappings of the whole plane, for some $K$ (we can take the bigger quasiconformal constant of both extensions if they differ). It is also well known that this yields that $f^{*},f\in Lip\left(\frac{1}{K}\right)$. On the other hand $f^{*},f\in Lip\left(\frac{1}{K}\right)$ yields $\Gamma\in LS(K)$ (see e.g. \cite{DD16}). Thus $K=2$ is a threshold value if we want $\Gamma$ to be the minimum set of a strictly subharmonic function.
	
	On the other hand Smirnov (See \cite{Smi10}), proving a conjecture of Astala (See \cite{Ast94}) obtained the upper bound 	$1+\left(\frac{K-1}{K+1}\right)^{2}$ for the Hausdorff dimension of a quasicircle obtained by a $K$-quasiconformal mapping. Since we want $K<2$, it turns out that the Hausdorff dimension of $\Pi_{\theta}$ can not exceed $\frac{10}{9}$ if $\Pi_{\theta}$ is the minimum set of a strictly subharmonic function. By the formula for the Hausdorff dimension of $\Gamma_{\theta}$, this corresponds to $\theta=\arccos(2^{-\frac{1}{10}})$.
	\end{proof}
	\begin{remark}
		Geometric intuition suggest that the threshold $\theta$ should be $\frac{\pi}{8}$ (because the approximating curves of $\Gamma_\theta$ have only obtuse angles if $\theta<\frac{\pi}{8}$). However $\arccos(2^{-\frac{1}{10}})<\frac{\pi}{8}\approx 0.392699$. In particular this shows that the \L ojasiewicz-Siciak exponent and the property of being a minimum set of a strictly (pluri)subharmonic function are not stable under passing to the limit in the Hausdorff metric of compact sets,  even if strong topological assumptions are imposed on the limit set.
	\end{remark}
	\begin{remark}
		We still, however, can not rule out the possibility that other compact sets in the plane, of higher Hausdorff dimension, can still be the minimum sets of strictly subharmonic functions. They however can not be quasicircles, and our methods do not apply to them. 
	\end{remark}
	
	\subsection{Viscosity theory of  (pluri)subharmonic functions}
Now we will prove that for small $\theta$, the curve $\Pi_{\theta}$ contains no points at which differential test from above exists for a certain subharmonic function.

  Let $$\tilde V_{\Pi_{\theta}}(z)=\begin{cases}
V^*_{\Pi_{\theta}}(z), \text{ if } z \text { lies in the unbounded component of the complement of } \Pi_{\theta} \\
0, \text{ if } z\in \Pi_{\theta}\\
V^*_{\iota(\Pi_{\theta})}(\frac{1}{z}), \text{ if } z \text { lies in the bounded component of the complement of } \Pi_{\theta} \\
\end{cases},$$
where $\iota(z)=\frac{1}{z}$ is the standard inversion. It is clear that $\tilde V_{\Pi_{\theta}}$ is subharmonic on $\mathbb C\setminus\{0\}$.
\begin{theorem}
	\label{difftest} For any $\theta$, such that $\Pi_{\theta}$  is a minimum set of a strictly subharmonic function, the set $\Pi_{\theta}$ contains only points at which the function  $\tilde V_{\Pi_{\theta}}$ allows no differential test from above. In particular such a set can have  Hausdorff dimension greater than $1$. 
\end{theorem} 

\begin{proof} Fix $z_{0}\in \Pi_{\theta}$ and suppose the contrary, that is, there exists a $q$ which is a differential test from above. Suppose that the neighborhood $V=B(z_{0},r)$ is a sufficiently small disc. First, since $\tilde V_{\Pi_{\theta}}(z_0)-q(z)=\sup_{w\in B(z_{0},r)}(\tilde V_{\Pi_{\theta}}(w)-q(w))$, we can assume that $q(z_{0})=0$ by subtracting a constant. Now since $q\geq V_{\Pi_{\theta}}\geq 0$, $q$ has a local minimum at $z_{0}$ and since $q\in\mathcal C^{2}$, $\nabla q(z_{0})=0$. It follows by the Taylor expansion that $q(z)\leq c|z-z_{0}|^2+o(|z-z_{0}|^{2})$ for some constant $c\geq 0$, that is, $ \tilde V_{\Pi_{\theta}}(z)\leq c|z-z_{0}|^2+o(|z-z_{0}|^{2})$ in $B(z_{0},r)$.
	
	We recall an estimate from \cite{DD16}:
	If $K$ is compact regular set such that $\hat{\mathbb C}\setminus K$ is simply connected  and $K\in LS(\alpha)$, for $\alpha<2$ then $\Delta V_{K}^{\frac{2}{\alpha}}(z)\geq \frac{2}{\alpha}\left(\frac{2}{\alpha}-1\right) \frac{V_{K}^{\frac{2}{\alpha}}}{dist(z, K)} \geq c_{1}>0,$ at every point where the Laplacian is well defined. This is satisfied both for $V^*_{\Pi_{\theta}}$ and $V^*_{\iota(\Pi_{\theta})}$ and hence for $\tilde V_{\Pi_{\theta}}$. The Laplacian of $(\tilde V_{\Pi_{\theta}})^{\frac{2}{\alpha}}$ is singular only on $\Pi_{\theta}$ which has planar Lebesgue measure $0$. Hence the measure integral $\int_{B(z_0,r)}\Delta \tilde V_{\Pi_{\theta}}^{\frac{2}{\alpha}}$ is equal to the Lebesgue integral $ \int_{B(z_0,r)\setminus \Pi_\theta}\Delta \tilde V_{\Pi_{\theta}}^{\frac{2}{\alpha}}d\lambda^2$.
	
	The Jensen formula gives us
	
	$$
	\frac1{2\pi}\int_0^{2\pi}\tilde V_{\Pi_{\theta}}^{\frac{2}{\alpha}}\left (z_{0}+re^{i\theta}\right)d\theta=\frac1{2\pi}\int_0^{2\pi}\tilde V_{\Pi_{\theta}}^{\frac{2}{\alpha}}\left (z_{0}+re^{i\theta}\right)d\theta-\tilde V_{\Pi_{\theta}}^{\frac{2}{\alpha}}\left (0\right)$$$$=
	\frac1{2\pi}\int_0^rs^{-1}\int_{B\left (z_{0},s\right)}\Delta \tilde V_{\Pi_{\theta}}^{\frac{2}{\alpha}}\left (z\right)dzds\\
	\geq \frac1{2\pi}\int_0^rs^{-1}\pi c_{1} s^2ds=c_{1}r^2/4.
	$$
	On the other hand
	$$\frac1{2\pi}\int_0^{2\pi}\tilde V_{\Pi_{\theta}}^{\frac{2}{\alpha}}\left (z_{0}+re^{i\theta}\right)d\theta\leq \frac1{2\pi}\int_0^{2\pi} q(z_{0}+re^{i\theta})^{\frac{2}{\alpha}}d\theta\leq\frac1{2\pi}\int_0^{2\pi}(cr^2+o(r^2))^{\frac{2}{\alpha}}d\theta$$$$= c_{2} r^{\frac{4}{\alpha}}+o (r^{\frac{4}{\alpha}}).$$
			Hence
		$$c_{1}r^2/4\leq c_{2} r^{\frac{4}{\alpha}}+o (r^{\frac{4}{\alpha}}). $$
		Now since $\alpha<2$ this is a contradiction for small enough $r$.
	\end{proof}
	\begin{remark}
		We did not use any properties of $\Pi_\theta$, besides $\Pi_\theta\in LS(\alpha)$. Hence our theorem applies to any compact sets $K\in LS(\alpha)$, for some $\alpha<2$, which are either Jordan curves or  connected and locally connected sets with empty interior that do not disconnect the plane. There are a lot of examples with Hausdorff dimension greater than $1$. 
	\end{remark}
 \begin{theorem}\label{sevrealdimensions} The set $E\subset\Omega\subset\CC$ of points, at which no differential test from above exists for a given $\varphi\in PSH(\Omega)$, can have Hausdorff dimension greater than $2n-1$.
 	
 \end{theorem}
 
 \begin{proof} Consider the function $\varphi(z_{1},\cdots,z_{n}):=\max\{\log|z_{1}\cdots z_{n-1}|,0\}+\tilde V_{\Pi_{\theta}}(z_{n})$. Clearly $\varphi\in PSH(\mathbb C^{n-1}\times(\mathbb C\setminus\{0\})$. Fix a point in $w=(w_{1},\cdots,w_{n})\in\mathbb D^{n-1}\times \Pi_{\theta}$, where $\mathbb D^{n-1}$ is the unit polydisc of dimension $n-1$, and suppose that $\varphi$ allows a differential test  from above $q$ at $w$. Then $\varphi$ restricted to the line $l:\mathbb C\setminus\{0\}\ni\lambda\to (w_{1},\cdots,w_{n-1},\lambda)$ is subharmonic and $\varphi\circ l= \tilde V_{\Pi_{\theta}}$. Moreover, the  $q\circ l$ is a differential test from above for $\varphi\circ l$ at $w_{n}$. The assertion now follows from Theorem \ref{difftest} and the fact that the Hausdorff dimension of $\mathbb D^{n-1}\times\Pi_{\theta}$ is $2(n-1)+\frac{\log 2}{\log (2\cos \theta)}$.
 	\end{proof}
	
	
	
	The next theorem aims to provide an answer to a problem posed in \cite{DGZ16}:
	\begin{theorem}\label{trudinger} Let $\Omega\subset \CC$ be a domain and let $\varphi\in PSH(\Omega)$. Then almost everywhere in $\Omega$ (with respect to  the Lebesgue measure) $\varphi$ allows a differential test from above.
		\end{theorem}
	\begin{proof}	
		The proof belongs to Trudinger, \cite{Tru89}. We present it here in the complex setting for the reader's convenience.

		First, we reduce the problem to the case when $\varphi$ is bounded.
		
		 Take a general $\psi\in PSH(\Omega)$. The function $e^{\psi}$ is still plurisubharmonic an bounded below. For any differential test from above  $q$  for $e^{\psi}$ at the point $w\in\Omega$, $\log( q+\Vert z-w\Vert^2)$ is a differential test from above for $\psi$, except at points $w$ where $e^{\psi}$ vanishes. These points however constitute the set $\{\psi=-\infty\}$, which is thus pluripolar and hence of Lebesgue measure zero. Finally we can exhaust $\Omega$ with open sets $\Omega_{n}=\{z\in\Omega| e^{\psi(z)}<n\}$. Note that $e^{\psi}$ is a bounded plurisubharmonic function in $\Omega_{n}$. In particular the Monge-Amp\`{e}re operator is well-defined on it (see \cite{BT82}) and $(dd^c e^{\psi})^n$ can be interpreted as a non negative Borel measure.
		
		Denote by $E$, $E\subset\Omega$ the set where $\varphi$ allows a differential test from above. First we note that
			\begin{equation}\label{mink}(dd^{c}(r+\eta))^{n}-(dd^{c}r)^n\geq (dd^{c}\eta)^n\geq 0,\end{equation}
			as measures, for any plurisubharmonic $r$ and $\eta$ in the domain of definition of the Monge-Amp\`{e}re operator.
		
		
		If $\varphi$ is bounded on $\Omega$ then denote by $\omega_{0}:= \sup_{\tilde w\in \Omega}\varphi(\tilde w
		)-\inf_{w\in \Omega}\varphi(w)$ the oscillation of $\varphi$. For $\varepsilon>0$ we define $\Omega_{\varepsilon\sqrt{2\omega_0}}=\{z\in\Omega| dist (z,\partial\Omega)>\varepsilon\sqrt{2\omega_0}\}$ and
		$$\varphi_{\varepsilon}(z)=\sup_{w\in\Omega_{\varepsilon\sqrt{2\omega_0}}}\left(\varphi(w)-\frac{\Vert w-z\Vert^2}{2\varepsilon^2}\right).$$
		
		The supremum is attained at some point $\tilde z\in\Omega$ such that $\Vert z-\tilde z\Vert<\varepsilon\sqrt{2\omega_{0}}$, assuming that $dist (z,\partial\Omega)>\varepsilon\sqrt{2\omega_{0}}$. Hence for $z\in \Omega_{\varepsilon\sqrt{2\omega_0}}$ the supremum can be taken over the whole $\Omega$. It is a standard fact that the functions $\varphi_{\varepsilon}$ decrease to $\varphi$ as $\varepsilon\to 0^{+}$. Moreover they are  Lipshitz and $\varphi_{\varepsilon}(z)+k\Vert z\Vert ^2$ is a convex function in $\Omega_{\varepsilon\sqrt{2\omega_0}}$ for any $k\geq\frac{1}{2\varepsilon^2}$. The latter property means that $\varphi_{\varepsilon}$ is semi-convex and hence twice differentiable almost everywhere by the classical theorem of Alexandrov. Also   $\varphi_{\varepsilon}$ is plurisubharmonic in $\Omega_{\varepsilon\sqrt{2\omega_0}}$ (see \cite{EGZ11}) and hence $(dd^c\varphi_{\varepsilon})^{n}\geq 0$ as measure.
		
		
		
		Fix a ball $B(z,R)\subset \Omega$ with $R<dist(z,\partial \Omega)$. For $\varepsilon$ small enough we have $B(z,R)\subset\Omega_{\varepsilon\sqrt{2\omega_0}}$. We restrict further considerations to such small $\varepsilon$. For $k\in \mathbb N\setminus\{0\}$ let 
		
		$$\varphi_{k,\varepsilon}(w)=\varphi_{\varepsilon}-k (\Vert w-z\Vert^2-R^2).$$
		For any $k$, $\varphi_{k,\varepsilon}$ agrees with $\varphi_{\varepsilon}$ on $\partial B(z,R)$.

		Let $E_{k,\varepsilon}$ be the upper contact set of $\varphi_{k,\varepsilon}$, that is the set of points, where the graph of $\varphi_{k,\varepsilon}$ lies below a supporting hyperplane in $\mathbb R^{2n+1}$, $$E_{k,\varepsilon}=\{w\in B(z,r)| \varphi_{k,\varepsilon}(\tilde w)\leq \varphi_{k,\varepsilon}(w)+\langle p,\tilde w -w\rangle \},$$ for some $p\in \mathbb R^{2n}$, depending on $w$, and  any  $\tilde w  \in B(z,r)$ (here $\langle\cdot,\cdot\rangle$ denotes the real inner product in $\mathbb R^{2n}$ and $z$ and $w$ are regarded as points in $\mathbb R^{2n}=\mathbb R^{n}+i\mathbb R^{n}\cong  \CC$). The set $E_{k,\varepsilon}$ is relatively closed in $B(z,r)$ and hence measurable.  Note that by construction $-\varphi_{k,\varepsilon}$ would be convex on every open subset of  $E_{k,\varepsilon}$. Thus $(dd^{c} (-\varphi_{k,\varepsilon}))^{n}\geq 0$ as measure on $E_{k,\varepsilon}$. 
		
		Likewise we introduce the upper contact set of $\varphi-k(\Vert w-z\Vert^2-R^2)$:
		$$E_{k}=\{w\in B(z,r)| \varphi(\tilde w)-k(\Vert \tilde w-z\Vert^2-R^2)\leq \varphi(w)-k(\Vert w-z\Vert^2-R^2)+\langle p,\tilde w -w\rangle \}.$$
		Note that $E_{k}\subset E$, since by construction $q(\tilde w)=k(\Vert \tilde w-z\Vert^2-R^2)+\langle p,\tilde w-z \rangle$ is a differential test from above at $w\in E_{k}$. Also the sets $E_{k,\varepsilon}$ approximate $E_{k}$ as $\varepsilon\to 0^{+}$, hence $$\limsup_{\varepsilon\to 0^{+}}\lambda ^{2n}(E_{k,\varepsilon})\leq \lambda ^{2n}(E_{k})\leq \lambda ^{2n}(E\cap B(z,R)).$$
		
		 Using (\ref{mink}) we have	
		$$ (dd^c (-\varphi_{k,\varepsilon}))^n\leq (dd^c k(\Vert w-z\Vert^2-R^2))^n- (dd^c \varphi_{\varepsilon})^n\leq (dd^c k(\Vert w-z\Vert^2-R^2))^n=4^{n}n!k^{n}d\lambda^{2n}.$$
		Now by the Alexandrov maximum principle
		\begin{equation}
		\label{alexandrov}
		\sup _{B(z,R)} \varphi_{k,\varepsilon}\leq \sup_{\partial B(z,R)}\varphi_{k,\varepsilon}+Cdiam B(z,R)\left(\int_{E_{k,\varepsilon}}|\det D^2 (\varphi_{k,\varepsilon})|\right)^{\frac{1}{2n}}\end{equation}$$\leq \sup_{\partial B(z,R)}\varphi_{\varepsilon}+2C_{1}R\left(\int_{E_{k,\varepsilon}} fd\lambda^{2n}\right)^{\frac{1}{2n}}\leq \sup_{\partial B(z,R)}\varphi_{\varepsilon}+C_{2}Rk(\lambda^{2n}(E_{k,\varepsilon}))^{\frac{1}{2n}},$$
		where $\sqrt{f}d\lambda^{2n}=(dd^c (-\varphi_{k,\varepsilon}))^n$. The second inequality follows from the estimate (see \cite{Blo05}) $\det\left( \frac{\partial^2}{\partial z_{i}\partial \bar z_{j}}u\right)\geq 2^{-n} \sqrt{\det D^2 u}=2^{-n} \sqrt{|\det D^2 (-u)|}$ for any  $u$ at any point where it is twice differentiable and it's real Hessian is positive definite.  
		

		If $k$ is big enough (\ref{alexandrov}) gives us
		$$kR^2+\varphi(z)=\sup_{B(z,R)} \varphi_{k,\varepsilon}\leq \sup_{\partial B(z,R)}\varphi_{\varepsilon}+C_{2}Rk(\lambda^{2n}(E_{k,\varepsilon}))^{\frac{1}{2n}}.$$
		Hence passing with $\varepsilon \to 0^{+}$
		$$R^2\leq \frac{1}{k}\omega_{0}+CR(\lambda ^{2n}(E_{k}))^{1/(2n)}.$$ 
		Finally, $$\frac{\lambda^{2n}(E\cap B(z,R))}{\lambda^{2n}(B(z,R))}\geq\limsup_{k\to\infty}\frac{\lambda^{2n}(E_{k})}{\lambda^{2n}(B(z,R))}\geq \limsup_{k\to\infty}\frac{(\frac{R}{C}-\frac{\omega_{0}}{kCR})^{2n}}{\frac{\pi^n}{n!}R^{2n}}=\frac{n!}{\pi^nC^{2n}}>0.$$
		
		This means that the Lebesgue density at $z$ of the set $E$, that is $\lim_{R\to0^{+}}\frac{\lambda^{2n}(E\cap B(z,R))}{\lambda^{2n}(B(z,R))}$, is not $0$ for any $z\in\Omega$. On the other  hand  it must be $0$  almost everywhere in $\Omega\setminus E$ by the Lebesgue's density theorem. Now it is clear that $E$ is of full measure.
		\end{proof}
		
		Finally, we want to show a method of producing differential test from above for certain plurisubharmonic functions at points where no differentiability is assumed. 
		\begin{theorem}\label{blocki} Let $U\subset \CC$ be a domain that contains the ball $$B(z_{0},R)=\{z\in\CC|  \Vert z-z_{0}\Vert\leq R\}.$$ Assume that $u\in PSH(U)$ obeys the conditions:
			\begin{enumerate}
				\item For some $\Lambda>0$  it holds that $\Lambda d\lambda^{2n}\geq (dd^cu)^{n}\geq 0$ on $B(z_{0},R)$ as measures, where $\lambda^{2n}$ is the Lebesgue measure.
				\item There exists $\sigma>0$ such that  $u(z)\geq \sigma ||z-z_{0}||^2$ when $z\in B(z_{0},R)$ and $u(z_{0})=0$.
			\end{enumerate}
			Then, there exists a constant $c=c(n,\Lambda)$ such that $u(z)\leq \frac{c}{\sigma^{2n-1}}\Vert z-z_{0}\Vert^{2}$ for all 
			$z\in B(z_{0}, R)$.
		\end{theorem}
		\begin{proof}
			Without loss of generality we assume that $z_{0}=0$. Denote by $T_{r}=\sup_{\Vert z\Vert =r}u(z)$ and by $t_{r}=\inf_{\Vert z\Vert =r}u(z)$ for any $0<r\leq R$. By the second condition one  has that
			\begin{equation}
			\label{equ1} t_{r}\geq\sigma r^{2}.
			\end{equation} Let also $U_{r}$ denote $\{z\in U\cap B(0,R)| u(z)<r\} $. Again by the second assumption
			\begin{equation} 
			\label{equ2}
			U_{r}\subset \{z\in U| \sigma \Vert z\Vert^2<r\}=B\left(0, \sqrt{\frac{r}{\sigma}}\right)\cap U.
			\end{equation}
			and hence 
			\begin{equation}
			\label{equ3}
			Vol(U_{r})\leq Vol \left(B\left(0, \sqrt{\frac{r}{\sigma}}\right)\right)=c_{n}\frac{r^{n}}{\sigma^{n}},
			\end{equation}
			where $c_{n}=\frac{\pi^n}{n!}$ is the volume of the unit ball in $\CC$. Observe also that if $\sqrt{\frac{r}{\sigma}}<R$, which is equivalent to $r<\sigma R^2$, then $U_{r}$ is relatively compact in $int B(0,R)$ and hence in $U$.
			
			We will use the following result from (\cite{Blo05}, Corollary $2$):
			Let $\Omega$ be a bounded domain in $\CC$. Assume that $v\in PSH(\Omega)$ has the boundary value zero ($v=0$ on $\partial\Omega$). Let also $(dd^c v)^{n}=fd\lambda^{2n}$. Then there exists a constant $\tilde c_{n}$ depending only on the dimension, so that
			\begin{equation}
			\label{equ4}
			\Vert v\Vert_{L^{\infty}(\Omega)}\leq \tilde c_{n}diam(\Omega)(Vol(\Omega))^{\frac{1}{2n}}\Vert f\Vert_{L^{\infty}(\Omega)}^{\frac{1}{n}}.
			\end{equation}
			The constant $\tilde c_{n}$ can be taken to be $\frac{2(n!)^{\frac{1}{2n}}}{\sqrt{\pi}}$. Note that in (\cite{Blo05}) there are additional assumptions on the regularity of $v$, namely that $v\in\mathcal C^{2}(\Omega)$, but they are not really necessary.
			
			
			We take $\Omega=U_{t_r}\cap B(0,r)$ and $v=u-t_r$. Since $u(0)=0$, $\Omega$ is nonempty and open in $B(0,r)$. Also since $u\geq t_r$ on $\partial B(0,r)$, we have   $v=0$ on $\partial \Omega$. Hence by (\ref{equ4}):
			\begin{equation}
			\label{equ5}
			t_r\leq \tilde c_{n}diam (\Omega) Vol(\Omega)^{\frac{1}{2n}} \Lambda^{\frac{1}{n}}\leq \tilde c_{n}2r (c_n r^{2n})^{\frac{1}{2n}} \Lambda^{\frac{1}{n}}
			\end{equation}
			Hence
			\begin{equation}
			\label{equ6} t_r\leq \tilde c r^{2},
			\end{equation}
			where $\tilde c=2\tilde c_{n}c_{n}^{1/2n}\Lambda^{1/n}$. In particular we see that if $\sigma> \tilde c$, then there are no functions $u$ satisfying the hypothesis of our theorem.

			We will also use (\cite{Blo05}, Proposition $3$, again the regularity assumptions are redundant) which says:
			
			If $\Omega\subset\CC$ is a bounded domain and $v\in PSH(\Omega)$ is negative, and such that $(dd^c v)^n=fd\lambda^{2n}$, then for any $a>0$, such that the level set $\{z\in \Omega| v(z)<\inf_{w\in\Omega} v(w) +a\}$ is nonempty and relatively compact in $\Omega$, then 
			\begin{equation}
			\label{equ7}
			\Vert v\Vert_{L^{\infty}(\Omega)}\leq a+ \left(\frac{\tilde c_{n}diam(\Omega)}{a}\right)^{2n}\Vert v\Vert_{L^{1}(\Omega)}\Vert f\Vert_{L^{\infty}(\Omega)}^{2}.
			\end{equation}
			The constant $\tilde c_{n}$ is the same as above.

			Now we consider $\Omega$ to be the open ball $int B(0,r)$ and we put $v=u-T_{r}$. Clearly $v$ is negative, $0\leq f\leq \Lambda$ and for any $0<a<t_{r}$, the set $\{z\in int B(0,r)| v(z)\leq -T_{r}+a\}$ is nonempty and relatively compact in $int B(0,r)$.

			Hence
			\begin{equation}
			\label{equ8}
			T_{r}=\Vert v\Vert_{L^{\infty}(B(0,r))}\leq a+ \left(\frac{\tilde c_{n}diam(B(0,r))}{a}\right)^{2n}\Vert v\Vert_{L^{1}(B(0,r))}\Vert f\Vert_{L^{\infty}(B(0,r))}^{2}\end{equation}
			$$\leq a+\left(\frac{\tilde c_{n}2r}{a}\right)^{2n}\Vert v\Vert_{L^{1}(B(0,r))}\Lambda^2.
			$$
			Now by using the co-area formula this can be further estimated by
			
			$$a+\left(\frac{\tilde c_{n}2r}{a}\right)^{2n}\int_{T_{r}-t_{r}}^{T_{r}}Vol(\{z\in B(0,r)|T_{r}-u>s\})ds\Lambda^2.$$
			Now by (\ref{equ3}) this is not greater than
			
			$$a+\left(\frac{\tilde c_{n}2r}{a}\right)^{2n}\int_{T_{r}-t_{r}}^{T_{r}}c_{n}\frac{(T_{r}-s)^n}{\sigma^n}ds\Lambda^2\leq a+c_{n}\left(\frac{\tilde c_{n}2r}{a}\right)^{2n}\frac{t_{r}^{n+1}}{\sigma^n}\Lambda^2.$$
			Passing with $a\to t_{r}^{-}$ we obtain in (\ref{equ8}):
			
			$$T_{r}-t_{r}\leq c_{n} \frac{\left(\tilde c_{n}2r\right)^{2n}} {\sigma^nt_{r}^{n-1}}\Lambda^2.$$
			Using (\ref{equ1}) this turns into
			
			$$\sigma^{n-1}r^{2n-2}(T_{r}-t_{r})\leq \frac{c_{n}\Lambda^{2}\tilde c_{n}^{2n}4^{n} r^{2n}}{\sigma^{n}}.$$
			Thus
			$$T_{r}\leq t_{r}+ \frac{c_{n}\Lambda^{2}\tilde c_{n}^{2n}4^{n} r^{2}}{\sigma^{2n-1}}\leq \frac{c}{\sigma^{2n-1}}r^2.$$
			Combining this with (\ref{equ6}) we get
			$$T_{r}\leq \left(\tilde c+\frac{c}{\sigma^{2n-1}}\right)r^2.$$
			For small enough $\sigma$, the term $\tilde c$ can be dropped.
			Now by manipulating $0<r<R$ we obtain the Theorem.
		\end{proof}
		
\medskip
{\bf Acknowledgments}
The first named author was partially supported by NCN grant no 2013/11/D/ST1/02599. The second named author was partially supported by  the Ideas Plus grant  0001/ID3/2014/63 of the Polish Ministry of Science
and Higher Education.


\begin{thebibliography}{KKPS01}
	\bibitem[AL11]{AL11} S.
	Akiyama, B.	Loridant: Boundary parametrization of self-affine tiles. J. Math. Soc. Japan.
	{\bf 63}, (2011), no. 2, 525-579.
	\bibitem[Ast94]{Ast94} K. Astala:	Area distortion of quasiconformal mappings. 
	Acta Math. {\bf 173} (1994), no. 1, 37-60. 
	\bibitem[BT82]{BT82} E. Bedford, B. A.Taylor: A new capacity for plurisubharmonic functions. Acta	Math. {\bf 149} (1982), 1-41.
	\bibitem[Be89]{Be89} V.I. Belyi: Moduli of continuity of exterior and interior conformal mappings of the unit disc. Ukrain. Math. J. {\bf 41} (1989), no 4, 409-414.
\bibitem[BJPP97]{BJPP97} C. Bishop; P. Jones; R. Pemantle; Y. Peres:
The dimension of the Brownian frontier is greater than $1$. 
J. Funct. Anal.  {\bf 143}  (1997),  no. 2, 309-336.
\bibitem[Blo05]{Blo05} Z.B\l ocki:  On uniform estimate in Calabi-Yau theorem, Proceedings of SCV2004, Beijing, Sci. China Ser. A Math. {\bf48} Supp. (2005), 244-247.
\bibitem[Caf90]{Caf90} L. Caffarelli: Interior $W^{2,p}$ estimates for solutions of the Monge-Amp\`ere equation, Ann. of Math. {\bf 131} (1990), 135-150.
\bibitem[CP11]{CP11}  M. Cantrell, J. Palagallo, Self intersection points of generalized Koch curves, Fractals {\bf 19}, 213 (2011), no.2, 213-220.
\bibitem[CL83]{CL83} M.G. Crandall, P.-L. Lions:   Viscosity solutions of Hamilton-Jacobi equations, Trans. Amer. Math. Soc. {\bf277} (1983), no 1, 1-42. 
\bibitem[DD16]{DD16}  S. Dinew, \.Z. Dinew: The minimum sets and free boundaries of strictly plurisubharmonic functions, 2016, arXiv:1604.02826.
\bibitem[DGZ16]{DGZ16} S. Dinew, V. Guedj, A. Zeriahi: Open Problems in Pluripotential Theory. Complex Var. and Ell. Eq. {\bf 61} (2016), no.7, 902-930.
\bibitem[EGZ11]{EGZ11} P. Eyssidieux, V. Guedj, A. Zeriahi:  Viscosity Solutions to Complex Monge-Amp\`{e}re Equations, Comm. Pure Appl. Math. {\bf64} (2011), no. 8, 1059-1094.
\bibitem[GH99]{GH99} M. Ghamsari, D. A. Herron: Bilipschitz homogeneous Jordan curves. Trans. Amer. Math. Soc. {\bf 351} (1999), no. 8, 3197-3216.
\bibitem[Gut]{Gut} C. Guti\'{e}rrez: The Monge-Amp\`ere Equation, Progress in Nonlinear Differential Equations and their Applications, {\bf 44 }. Birkh\"auser Boston, Inc., Boston, MA, 2001. xii+127 pp. ISBN: 0-8176-4177-7.
\bibitem[HW73]{HW73} F.R. Harvey, R.O. Wells:
Zero sets of non-negative strictly plurisubharmonic functions. 
Math. Ann.  {\bf201}  (1973), 165-170. 
\bibitem[IW15]{IW15} A. Iseli, K. Wildrick: Iterated function system quasiarcs,  	arXiv:1512.04688.
\bibitem[Le79]{Le79} F. Lesley: On Interior and Conformal Mappings of the Disk.  J. London Math. Soc. {\bf s2-20} (1979), no. 1, 67-78. 
\bibitem[Le82]{Le82} F. Lesley:
H\"older continuity of conformal mappings at the boundary via the strip method. 
Indiana Univ. Math. J.  {\bf 31}  (1982), no. 3, 341-354.
\bibitem[Le83]{Le83} F. Lesley: Domains with Lipschitz mapping Functions.
	 Ann. Acad. Sci. Fenn. Series A. I. Math.  {\bf8} (1983), 219-233.
\bibitem[NP80]{NP80} R. N\"akki; B. Palka:
Quasiconformal circles and Lipschitz classes. 
Comment. Math. Helv.  {\bf 55}  (1980), no. 3, 485-498.

\bibitem[Pi14]{Pi14} R. Pierzchała: The \L ojasiewicz-Siciak condition of the pluricomplex
Green function. Potential Anal. {\bf 40} (2014), no.1, 41-56.

 
 



\bibitem[Po93]{Po93} S.P. Ponomarev: On Hausdorff dimensions of quasiconformal curves. Siberian Math. J. {\bf34} (1993), no. 4, 717-722.
\bibitem[Po07]{Po07} S.P. Ponomarev: On some properties of Van Koch curves. Siberian Math. J. {\bf 48} (2007), no. 6, 1046-1059.
\bibitem[Ro01]{Ro01} S. Rohde: Quasicircles modulo bilipschitz maps. Rev. Mat. Iberoamer. {\bf 17} (2001), no. 3,  643-659.
\bibitem[Sa82]{Sa82} A. Sakai:
The intersection of the closures of two disjoint strongly pseudoconvex domains. Math. Ann.  {\bf 260}  (1982), no. 1, 117?118.
\bibitem[Smi10]{Smi10} S. Smirnov: Dimension of quasicircles. 
Acta Math. {\bf 205} (2010), no. 1, 189-197.
Stout, Edgar Lee(1-WA)
\bibitem[Sto07]{Sto07} E.L. Stout: Polynomial Convexity. 
Progress in Mathematics, {\bf 261}. Birkh\"auser Boston, Inc., Boston, MA, 2007. xii+439 pp. ISBN: 978-0-8176-4537-3; 0-8176-4537-3.
\bibitem[Tru89]{Tru89} N.S. Trudinger: On the twice differentiability of viscosity solutions of nonlinear elliptic equations. Bull. Austral. Math. Soc. {\bf 39} (1989), no. 3, 443-447.

\bibitem[Zer13]{Zer13} A. Zeriahi: A viscosity approach to degenerate complex Monge-Amp\`{e}re equations, Ann. Fac. Sci. Toulouse Math. (6) {\bf 22} (2013),
no. 4, 843-913. 

\end{thebibliography}
\end{document}